\documentclass[11pt]{elsarticle}
\usepackage{mathrsfs}

 \usepackage{graphics}

\usepackage{amssymb}

\usepackage{amsfonts}
\usepackage{pifont}
\usepackage{amsmath}
\usepackage{latexsym}
\usepackage{bm}
\usepackage{amssymb}
\usepackage{amsthm}
\usepackage{graphics}
\usepackage{epstopdf}
\usepackage{tikz}
\def\bi{\begin{align}}
\def\bin{\begin{align*}}
\newtheorem{thm}{Theorem}[section]

\newtheorem{lem}[thm]{Lemma}

\newtheorem{defn}[thm]{Definition}
\newtheorem{rem}[thm]{Remark}

\newcommand{\si}{\sigma}

\newcommand{\dif}{\mathrm{d}}

\newcommand{\ba}{\begin{array}}
\newcommand{\ea}{\end{array}}

\begin{document}

\begin{frontmatter}



\title{High order symplectic integrators based on continuous-stage Runge-Kutta-Nystr\"{o}m methods}

\author[a,b]{Wensheng Tang\corref{cor1}}
\ead{tangws@lsec.cc.ac.cn} \cortext[cor1]{Corresponding author.}
\address[a]{College of Mathematics and Statistics,\\
    Changsha University of Science and Technology,  Changsha 410114, China}
\address[b]{Hunan Provincial Key Laboratory of \\
    Mathematical Modeling and Analysis in Engineering,  Changsha 410114, China}
\author[c]{Yajuan Sun\corref{cor1}}
\ead{sunyj@lsec.cc.ac.cn}
\address[c]{LSEC, Academy of Mathematics and Systems Science,\\
Chinese Academy of Sciences, Beijing  100190, China}
\author[d]{Jingjing Zhang}
\ead{jjzhang06@outlook.com}
\address[d]{School of Science, East China Jiaotong  University, Nanchang 330013, China
}
\author[]{}

\begin{abstract}

On the basis of the previous work by Tang \& Zhang (Appl. Math.
Comput. 323, 2018, p. 204--219), in this paper we present a more
effective way to construct high-order symplectic integrators for
solving second order Hamiltonian equations. Instead of analyzing
order conditions step by step as shown in the previous work, the new
technique of this paper is using Legendre expansions to deal with
the simplifying assumptions for order conditions. With the new
technique, high-order symplectic integrators can be conveniently
devised by truncating an orthogonal series.

\end{abstract}

\begin{keyword}
Continuous-stage Runge-Kutta-Nystr\"{o}m methods; Hamiltonian
systems; Symplectic integrators;  Legendre polynomial expansion;
Simplifying assumptions.

\end{keyword}

\end{frontmatter}

\section{Introduction}
\label{}



In the last few decades, geometric integrators for the numerical
solution of various differential systems have attracted much
attention in the field of scientific and engineering computations
\cite{BlanesFernado,Channels90sio,Feng95kfc,Fengqq10sga,hairerlw06gni,
Leimkuhlerr04shd,sanzc94nhp}. Such type of integrators are related
with the terminology ``geometric" because they are applied to
systems with geometric features.  Strictly speaking, they must
preserve at least one of geometric properties of the given system.
The most significant advantage of employing such integrators is that
they can not only correctly capture the qualitative behaviors of the
exact flow of the system in phase space, but also give rise to a
more accurate long-time integration than general-purpose methods
\cite{Benetting94oth,hairerlw06gni,Shang99kam,Tang94feo}.

As is well known, traditional numerical methods such as Runge-Kutta
(RK) methods, partitioned Runge-Kutta (PRK) methods and
Runge-Kutta-Nystr\"{o}m (RKN) methods play a prominent role in the
numerical discretization of ordinary differential equations (ODEs)
\cite{butcher87tna,hairernw93sod,hairerw96sod}. Many geometric
integrators can be established from these numerical methods
\cite{lasagni88crk,sanz88rkm,sun93coh,sun95coh,suris88otc,suris89ctg},
and they are preferable for practical use due to their elegant
formulations and standardized implementations
\cite{hairerlw06gni,sanzc94nhp}. Recently, as a generalized form of
these classical methods, numerical schemes with continuous stage
(which implies infinitely-many stages) have been proposed and
developed in the literature
\cite{butcher72ato,butcher87tna,hairer10epv,miyatake14aee,miyatakeb16aco,
Tangs12tfe,Tangs12ana,Tangs14cor,Tanglx16cos,Tangz18spc}. The
construction of such ``continuous" integrators seems much easier
than traditional methods with finite stages, as the Butcher
coefficients are taken as continuous functions and they are allowed
for orthogonal expansions
\cite{brugnanoit12asf,Tangs14cor,Tanglx16cos,Tangz18spc,Tang18ano}.
It turned out that with the help of continuous-stage approach we can
construct many classical integrators of arbitrarily-high order,
without needing to solve the tedious nonlinear algebraic equations,
usually associated with the order conditions of the numerical
methods, in terms of many unknown coefficients. Moreover, some
important geometric integrators can be readily derived with this new
approach, and the prototype integrators amongst them are: symplectic
methods for Hamiltonian systems, symmetric methods for reversible
systems, and energy-preserving methods for Hamiltonian (even
Poisson) systems
\cite{brugnanoit10hbv,Celledoni09mmoqw,cohenh11lei,miyatake14aee,
miyatakeb16aco,quispelm08anc,Tangs14cor,Tanglx16cos,Tangz18spc}.

It is worth mentioning that by using continuous-stage approach one
may get some interesting results. A significant instance is that no
RK methods are energy-preserving for general non-polynomial
Hamiltonian systems \cite{Celledoni09mmoqw}, but energy-preserving
csRK methods obviously exist
\cite{brugnanoit10hbv,brugnanoi16lmf,hairer10epv,miyatake14aee,
miyatakeb16aco,quispelm08anc,Tangs12ana,Tangs12tfe,Tangs14cor}. It
is shown in \cite{Tangs12tfe,Tangsc17dgm} that some Galerkin
variational methods can be closely connected to continuous-stage
(P)RK methods. In addition, continuous-stage approaches may promote
the investigation of conjugate symplecticity of energy-preserving
methods \cite{hairer10epv,hairerz13oco,Tangs14cor}.

Recently, the present author et al. \cite{Tangz18spc} have
investigated the construction of symplectic RKN-type integrators on
the basis of continuous-stage methods. It is shown in
\cite{Tangz18spc} that by analyzing order conditions one by one with
Legendre expansions, symplectic integrators of any order can be
constructed step by step. However, when the same approaches are used
for deriving higher-order integrators, one has to perform much more
complicated analyses and tedious calculations, as more and more
order conditions have to be considered. To address this difficulty,
in this paper we propose a more effective way to construct
high-order integrators by considering the simplifying assumptions
for order conditions. Actually, a similar technique has been
previously developed in \cite{Tangs14cor} for constructing RK-type
methods.

The paper is organized as follows: In Section 2, we introduce the
definition of RKN-type methods for solving second-order differential
equations;  In Section 3, the order theory will be discussed; In
Section 4, we expound our approach for constructing high-order
symplectic integrators; Finally, a few conclusions are drawn in
Section 5.

\section{Runge-Kutta-Nystr\"{o}m-type methods}

Consider the following initial value problem in the form of
second-order system,
\begin{align}\label{eq:second}
q''=f(t, q),\;\;q(t_0)=q_0,\;\;q'(t_0)=q'_0,
\end{align}
where $f:\mathbb{R}\times\mathbb{R}^{d}\rightarrow\mathbb{R}^{d}$ is
a sufficiently smooth vector-valued function. The most popular
numerical methods for solving \eqref{eq:second} are
Runge-Kutta-Nystr\"{o}m methods, which can be defined as follows.
\begin{defn}\cite{hairerlw06gni}\label{RKN:def}
A Runge-Kutta-Nystr\"{o}m (RKN) method for solving \eqref{eq:second}
is defined by
\begin{subequations}
\begin{alignat}{4}
\label{eq:grkn1}
Q_i&=q_0 +hc_i q'_0 +h^2\sum\limits_{j=1}^s \bar{a}_{i j}f(t_0+c_j h, Q_j),\;i=1,\cdots, s, \\
\label{eq:grkn2}
q_{1}&=q_0+ h q'_0+h^2\sum\limits_{i=1}^{s}\bar{b}_if(t_0+c_i h, Q_i), \\
\label{eq:grkn3} q'_{1}& = q'_0 +h\sum\limits_{i=1}^{s} b_i
f(t_0+c_i h, Q_i),
\end{alignat}
\end{subequations}
which can be characterized by the following Butcher tableau
\[\ba{c|c}c & \bar{A}\\[4pt]
\hline & \bar{b} \\ \hline\\[-15pt] & b \ea\]
where $\bar{A}=(\bar{a}_{i j})_{s\times
s},\;\bar{b}=(\bar{b}_1,\cdots,\bar{b}_s),\;
b=(b_1,\cdots,b_s),\;c=(c_1,\cdots,c_s)$.
\end{defn}

In a similar manner, one defines continuous-stage
Runge-Kutta-Nystr\"{o}m methods.

\begin{defn}\cite{Tangz18spc}\label{csRKN:def}
Let $\bar{A}_{\tau, \si}$ be a function of variables $\tau,
\si\in[0,1]$ and $\bar{B}_\tau,\;B_\tau,\;C_\tau$ be functions of
$\tau\in[0,1]$. A continuous-stage Runge-Kutta-Nystr\"{o}m (csRKN)
method for solving \eqref{eq:second} is given by
\begin{subequations}
    \begin{alignat}{2}
    \label{eq:csrkn1}
&Q_\tau=q_0 +hC_\tau q'_0 +h^2\int_{0}^{1} \bar{A}_{\tau, \si} f(t_0+C_\si h, Q_\si) \dif \si, \;\;\tau \in[0, 1], \\
    \label{eq:csrkn2}
&q_{1}=q_0+ h q'_0+h^2 \int_{0}^{1} \bar{B}_\tau  f(t_0+C_\tau h, Q_\tau) \dif\tau, \\
    \label{eq:csrkn3}
&q'_1 = q'_0 +h\int_{0}^{1} B_\tau f(t_0+C_\tau h, Q_\tau) \dif\tau,
    \end{alignat}
\end{subequations}
which can be characterized by the following Butcher tableau
\[\ba{c|c} C_\tau & \bar{A}_{\tau,\si}\\[4pt]
\hline & \bar{B}_\tau \\ \hline\\[-15pt] & B_\tau \ea\]
\end{defn}

We shall refer to the methods in Definition \ref{RKN:def} and
\ref{csRKN:def} as ``RKN-type methods".

\section{Order theory for RKN-type methods}

\begin{defn}\cite{hairernw93sod}
A RKN-type method has order $p$, if for any sufficiently regular
problem \eqref{eq:second}, as $h\rightarrow0$, its \emph{local
error} satisfies
\begin{equation*}
q_1-q(t_0+h)=\mathcal{O}(h^{p+1}),\quad
q'_1-q'(t_0+h)=\mathcal{O}(h^{p+1}).
\end{equation*}
\end{defn}

A ``modern" order theory for RKN methods can be found in
\cite{hairernw93sod,hairerlw06gni,sanzc94nhp} and references
therein. However, in this section we do not plan to review all
aspects of the order theory, but the elegant parts in terms of
simplifying assumptions for order conditions will be picked up and
then extended for csRKN methods.

\subsection{Order theory for RKN methods}

In order to reduce the difficulty of analyzing the order accuracy,
the following simplifying assumptions for order conditions were
proposed \cite{hairernw93sod,hairerlw06gni}
\begin{equation*}\label{RKN-simpl-assump}
\begin{split}
&B(\xi):\; \sum_{i=1}^sb_ic_i^{\kappa-1}=\frac{1}{\kappa},\;1\leq\kappa\leq\xi,\\
&CN(\eta):\;\sum_{j=1}^s\bar{a}_{ij}c_j^{\kappa-1}=\frac{c_i^{\kappa+1}}{\kappa(\kappa+1)},
\;1\leq i\leq s,\,1\leq\kappa\leq\eta-1,\\
&DN(\zeta):\;\sum_{i=1}^sb_ic_i^{\kappa-1}\bar{a}_{ij}=
\frac{b_jc_j^{\kappa+1}}{\kappa(\kappa+1)}-\frac{b_jc_j}{\kappa}+\frac{b_j}{\kappa+1},\;1\leq
j\leq s,\,1\leq\kappa\leq\zeta-1.
\end{split}
\end{equation*}
\begin{thm}\cite{hairernw93sod}\label{ord_RKN}
If the coefficients of the RKN method
\eqref{eq:grkn1}-\eqref{eq:grkn3} satisfy the simplifying
assumptions $B(p),\,CN(\eta),\,DN(\zeta)$, and if
$\bar{b}_i=b_i(1-c_i)$, for all $i=1,\ldots, s$, then the method has
order at least $\min\{p,\,2\eta+2,\eta+\zeta\}$.
\end{thm}

\subsection{Order theory for csRKN methods}

Similarly to the classical case, we propose the following
simplifying assumptions:
\begin{equation*}\label{csRKN-simpl-assump}
\begin{split}
&\mathcal{B}(\xi):\quad \int_0^1B_\tau C_\tau^{\kappa-1}\,\dif
\tau=\frac{1}{\kappa},\;\; 1\leq\kappa\leq\xi,\\
&\mathcal{CN}(\eta):\quad
\int_0^1\bar{A}_{\tau,\,\sigma}C_\sigma^{\kappa-1}\,\dif
\sigma=\frac{C_\tau^{\kappa+1}}{\kappa(\kappa+1)},\;\;1\leq\kappa\leq\eta-1,\\
&\mathcal{DN}(\zeta):\quad \int_0^1B_\tau C_\tau^{\kappa-1}
\bar{A}_{\tau,\,\sigma}\,\dif \tau=\frac{B_\sigma
C_\sigma^{\kappa+1}}{\kappa(\kappa+1)}-\frac{B_\sigma
C_\sigma}{\kappa} +\frac{B_\sigma}{\kappa+1},\;\;
1\leq\kappa\leq\zeta-1,
\end{split}
\end{equation*}
where $ \tau,\,\si\in[0,1]$.
\begin{thm}\label{ord_csRKN}
If the coefficients of the csRKN method
\eqref{eq:csrkn1}-\eqref{eq:csrkn3} satisfy the simplifying
assumptions
$\mathcal{B}(p),\,\mathcal{CN}(\eta),\,\mathcal{DN}(\zeta)$, and if
$\bar{B}_\tau=B_\tau(1-C_\tau)$, for $\tau\in[0,1]$, then the method
has order at least $\min\{p,\,2\eta+2,\eta+\zeta\}$.
\end{thm}
\begin{proof}
This result can be proved similarly to the classical result
\cite{hairernw93sod}.
\end{proof}

We firstly introduce the $\iota$-degree normalized shifted Legendre
polynomial $P_\iota(t)$ by using the Rodrigues' formula
$$P_0(t)=1,\;P_\iota(t)=\frac{\sqrt{2\iota+1}}{\iota!}\frac{{\dif}^\iota}{\dif t^\iota}
\Big(t^\iota(t-1)^\iota\Big),\; \;\iota=1,2,3,\cdots.$$ A well-known
property of Legendre polynomials is that they are orthogonal to each
other with respect to the $L^2([0, 1])$ inner product
$$\int_0^1 P_\iota(t) P_\kappa(t)\,\dif t= \delta_{\iota\kappa},\quad\iota,\,
\kappa=0,1,2,\cdots,$$ and they satisfy the following integration
formulas
\begin{equation}\label{property}
\begin{split} &\int_0^xP_0(t)\,\dif
t=\xi_1P_1(x)+\frac{1}{2}P_0(x), \\
&\int_0^xP_\iota(t)\,\dif
t=\xi_{\iota+1}P_{\iota+1}(x)-\xi_{\iota}P_{\iota-1}(x),\quad
\iota=1,2,3,\cdots, \\
&\int_x^1P_\iota(t)\,\dif
t=\delta_{\iota0}-\int_0^{x}P_\iota(t)\,\dif t,\quad
\iota=0,1,2,\cdots,
\end{split}
\end{equation}
where $\xi_\iota=\frac{1}{2\sqrt{4\iota^2-1}}$ and
$\delta_{\iota\kappa}$ is the Kronecker delta. Hereafter we assume
$B_\tau=1,C_\tau=\tau$ \cite{Tangs14cor,Tanglx16cos}. Consequently,
the first assumption $\mathcal{B}(\xi)$ can be reduced to
\begin{equation*}
\int_0^1 \tau^{\kappa-1}\,\dif \tau=\frac{1}{\kappa},\quad
\kappa=1,\ldots,\xi,
\end{equation*}
which is obviously satisfied for any positive integer $\xi$. For
convenience, we denote this fact by $\mathcal{B}(\infty)$. In
addition, by taking the derivative with respect to $\tau$ and $\si$
respectively, it follows from $\mathcal{CN}(\eta)$ and
$\mathcal{DN}(\zeta)$
\begin{equation}\label{simpl-assump01}
\begin{split}
&\mathcal{CN}'(\eta):\quad
\int_0^1\frac{\dif}{\dif\tau}\bar{A}_{\tau,\,\sigma}\sigma^{\kappa-1}\,\dif
\sigma=\frac{\tau^{\kappa}}{\kappa}=\int_0^{\tau}\si^{\kappa-1}\,\dif \si,\quad1\leq\kappa\leq\eta-1,\\
&\mathcal{DN}'(\zeta):\quad \int_0^1 \tau^{\kappa-1}
\frac{\dif}{\dif\si}\bar{A}_{\tau,\,\sigma}\,\dif \tau=\frac{
\sigma^{\kappa}}{\kappa}-\frac{1}{\kappa}=-\int_\sigma^1\tau^{\kappa-1}\,\dif
\tau,\quad 1\leq\kappa\leq\zeta-1.
\end{split}
\end{equation}
Note that \eqref{simpl-assump01} do not imply $\mathcal{CN}(\eta)$
and $\mathcal{DN}(\zeta)$, hence we additionally assume
\begin{equation}\label{simpl-assump02}
\int_0^1\bar{A}_{0,\,\sigma}\sigma^{\kappa-1}\,\dif
\sigma=0,\quad1\leq\kappa\leq\eta-1,
\end{equation}
and
\begin{equation*}
\int_0^1\tau^{\kappa-1} \bar{A}_{\tau,\,0}\,\dif
\tau=\frac{1}{\kappa+1}=\int_0^1\tau^{\kappa}\,\dif \tau,\quad
1\leq\kappa\leq\zeta-1,
\end{equation*}
which in turn gives rise to
\begin{equation}\label{simpl-assump03}
\int_0^1\tau^{\kappa-1} (\bar{A}_{\tau,\,0}-\tau)\,\dif \tau=0,\quad
1\leq\kappa\leq\zeta-1.
\end{equation}

Since the Legendre polynomial sequence $\{P_\iota(t)\}$ forms a
complete orthogonal set in $L^2([0, 1])$, it allows us to consider
the following expansions (with $\tau$ and $\si$ being fixed
respectively)
\begin{equation}\label{expansion1}
\frac{\dif}{\dif\tau}\bar{A}_{\tau,\,\sigma}=\sum_{\iota\geq0}
\gamma_\iota(\tau) P_\iota(\sigma),\;\;\;
\frac{\dif}{\dif\si}\bar{A}_{\tau,\,\sigma}=\sum_{\iota\geq0}
\lambda_\iota(\sigma) P_\iota(\tau),
\end{equation}
where $\gamma_\iota(\tau),\,\lambda_\iota(\sigma)$ are unknown
coefficient functions. Observe that \eqref{simpl-assump01} imply
\begin{equation}\label{simpl-assump04}
\begin{split}
&\mathcal{CN}'(\eta):\quad
\int_0^1\frac{\dif}{\dif\tau}\bar{A}_{\tau,\,\sigma}P_{\kappa-1}(\si)\,\dif
\sigma=\int_0^{\tau}P_{\kappa-1}(\si)\,\dif \si,\quad1\leq\kappa\leq\eta-1,\\
&\mathcal{DN}'(\zeta):\quad \int_0^1 P_{\kappa-1}(\tau)
\frac{\dif}{\dif\si}\bar{A}_{\tau,\,\sigma}\,\dif
\tau=-\int_\sigma^1P_{\kappa-1}(\tau)\,\dif \tau,\quad
1\leq\kappa\leq\zeta-1,
\end{split}
\end{equation}
which leads to
\begin{equation}\label{expan_coe1}
\begin{split}
&\gamma_\iota(\tau)=\int_0^{\tau}P_\iota(\si)\,\dif \si,\quad0\leq\iota\leq\eta-2,\\
&\lambda_\iota(\sigma)=-\int_\sigma^1P_\iota(\tau)\,\dif \tau,\quad
0\leq\iota\leq\zeta-2.
\end{split}
\end{equation}
Substituting \eqref{expan_coe1} into \eqref{expansion1} and using
\eqref{property} gives
\begin{equation}\label{expan1}
\begin{split}
\frac{\dif}{\dif\tau}\bar{A}_{\tau,\,\sigma}
&=\sum_{\iota=0}^{\eta-2}\int_0^{\tau}P_\iota(x)\,\dif
x\,P_\iota(\sigma)+\sum_{\iota\geq\eta-1} \gamma_\iota(\tau)
P_\iota(\sigma)\\
&=\frac{1}{2}+\sum_{\iota=0}^{\eta-2}\xi_{\iota+1}
P_{\iota+1}(\tau)P_\iota(\sigma)-\sum_{\iota=0}^{\eta-3}\xi_{\iota+1}
P_{\iota+1}(\sigma)P_\iota(\tau)+\sum_{\iota\geq\eta-1}
\gamma_\iota(\tau) P_\iota(\sigma),
\end{split}
\end{equation}
\begin{equation}\label{expan2}
\begin{split}
\frac{\dif}{\dif\si}\bar{A}_{\tau,\,\sigma}
&=-\sum_{\iota=0}^{\zeta-2}\int_{\sigma}^1P_\iota(x)\,\dif
x\, P_\iota(\tau)+\sum_{\iota\geq\zeta-1}\lambda_\iota(\sigma)P_\iota(\tau)\\
&=-\frac{1}{2}-\sum_{\iota=0}^{\zeta-3}\xi_{\iota+1}
P_{\iota+1}(\tau)P_\iota(\sigma)+\sum_{\iota=0}^{\zeta-2}\xi_{\iota+1}
P_{\iota+1}(\sigma)P_\iota(\tau)+\sum_{\iota\geq\zeta-1}\lambda_\iota(\sigma)P_\iota(\tau).
\end{split}
\end{equation}
Integrating \eqref{expan1} with respect to $\tau$ and \eqref{expan2}
with respect to $\si$, yields
\begin{equation}\label{integ}
\begin{split}
\bar{A}_{\tau,\,\sigma}-\bar{A}_{0,\,\sigma}=&\frac{1}{2}\tau+\sum_{\iota=0}^{\eta-2}\xi_{\iota+1}
\int_0^{\tau}P_{\iota+1}(x)\,\dif
x\,P_\iota(\sigma)\\
&-\sum_{\iota=0}^{\eta-3}\xi_{\iota+1}
P_{\iota+1}(\sigma)\int_0^{\tau}P_\iota(x)\,\dif
x+\sum_{\iota\geq\eta-1} \int_0^{\tau}\gamma_\iota(x)\,\dif x
\,P_\iota(\sigma),\\
\bar{A}_{\tau,\,\si}-\bar{A}_{\tau,\,0}=&-\frac{1}{2}\sigma-\sum_{\iota=0}^{\zeta-3}\xi_{\iota+1}
P_{\iota+1}(\tau)\int_0^{\si}P_\iota(x)\,\dif
x\\
&+\sum_{\iota=0}^{\zeta-2}\xi_{\iota+1}\int_0^{\si}
P_{\iota+1}(x)\,\dif
x\,P_\iota(\tau)+\sum_{\iota\geq\zeta-1}\int_0^{\si}\lambda_\iota(x)\,\dif
x\,P_\iota(\tau).
\end{split}
\end{equation}
Taking into account \eqref{simpl-assump02} and
\eqref{simpl-assump03}, implies
\begin{equation}\label{simpl-assump05}
\begin{split}
&\int_0^1\bar{A}_{0,\,\sigma}P_{\kappa-1}(\sigma)\,\dif
\sigma=0,\quad1\leq\kappa\leq\eta-1,\\
&\int_0^1P_{\kappa-1}(\tau) (\bar{A}_{\tau,\,0}-\tau)\,\dif
\tau=0,\quad 1\leq\kappa\leq\zeta-1.
\end{split}
\end{equation}
We then consider the following orthogonal expansions
\begin{equation}\label{expansion2}
\bar{A}_{0,\,\sigma}=\sum_{\iota\geq0} \alpha_\iota
P_\iota(\sigma),\;\;\; \bar{A}_{\tau,\,0}-\tau=\sum_{\iota\geq0}
\beta_\iota P_\iota(\tau),
\end{equation}
where $\alpha_\iota,\,\beta_\iota$ are real numbers. By inserting
\eqref{expansion2} into \eqref{simpl-assump05} we get
\begin{equation}\label{expan_coe2}
\alpha_\iota=0,\;\; 0\leq\iota\leq\eta-2;\;\;\; \beta_\iota=0,\;\;
0\leq\iota\leq\zeta-2.
\end{equation}
Then, it follows
\begin{equation}\label{expan3}
\bar{A}_{0,\,\sigma}=\sum_{\iota\geq\eta-1} \alpha_\iota
P_\iota(\sigma),\;\;\;
\bar{A}_{\tau,\,0}=\tau+\sum_{\iota\geq\zeta-1} \beta_\iota
P_\iota(\tau).
\end{equation}
Using the known equality $\tau=\frac{1}{2}P_0(\tau)+\xi_1P_1(\tau)$
and inserting \eqref{expan3} into \eqref{integ}, gives
\begin{equation*}
\begin{split}
\bar{A}_{\tau,\,\sigma}=&\frac{1}{4}P_0(\tau)+\frac{1}{2}\xi_1P_1(\tau)+
\sum_{\iota=0}^{\eta-2}\xi_{\iota+1}
\int_0^{\tau}P_{\iota+1}(x)\,\dif
x\,P_\iota(\sigma)\\
&-\sum_{\iota=0}^{\eta-3}\xi_{\iota+1}
P_{\iota+1}(\sigma)\int_0^{\tau}P_\iota(x)\,\dif
x+\sum_{\iota\geq\eta-1}\big(\alpha_\iota+\int_0^{\tau}\gamma_\iota(x)\,\dif
x\big)\,P_\iota(\sigma),
\end{split}
\end{equation*}
\begin{equation*}
\begin{split}
\bar{A}_{\tau,\,\si}=&\frac{1}{4}P_0(\tau)+\xi_1P_1(\tau)-\frac{1}{2}\xi_1P_1(\si)
-\sum_{\iota=0}^{\zeta-3}\xi_{\iota+1}P_{\iota+1}(\tau)\int_0^{\si}P_\iota(x)\,\dif
x\\
&+\sum_{\iota=0}^{\zeta-2}\xi_{\iota+1}\int_0^{\si}
P_{\iota+1}(x)\,\dif
x\,P_\iota(\tau)+\sum_{\iota\geq\zeta-1}\big(\beta_\iota+\int_0^{\si}\lambda_\iota(x)\,\dif
x\big)\,P_\iota(\tau).
\end{split}
\end{equation*}
By exploiting \eqref{property} once again, it ends up with
\begin{equation*}
\begin{split}
\bar{A}_{\tau,\,\sigma}=&\frac{1}{6}-\frac{1}{2}\xi_1P_1(\si)+\frac{1}{2}\xi_1P_1(\tau)+
\sum_{\iota=1}^{\eta-3}\xi_{\iota}\xi_{\iota+1}
P_{\iota-1}(\tau)P_{\iota+1}(\si)\\
&-\sum_{\iota=1}^{\eta-2}\big(\xi_{\iota}^2+\xi_{\iota+1}^2\big)
P_{\iota}(\tau)P_{\iota}(\sigma)+\sum_{\iota=1}^{\eta-1}\xi_{\iota}\xi_{\iota+1}
P_{\iota+1}(\tau)P_{\iota-1}(\sigma)\\
&+\sum_{\iota\geq\eta-1}\phi_\iota(\tau)\,P_\iota(\sigma),
\end{split}
\end{equation*}
\begin{equation*}
\begin{split}
\bar{A}_{\tau,\,\si}=&\frac{1}{6}-\frac{1}{2}\xi_1P_1(\si)+\frac{1}{2}\xi_1P_1(\tau)
+\sum_{\iota=1}^{\zeta-1}\xi_{\iota}\xi_{\iota+1}
P_{\iota-1}(\tau)P_{\iota+1}(\si)\\
&-\sum_{\iota=1}^{\zeta-2}\big(\xi_{\iota}^2+\xi_{\iota+1}^2\big)
P_{\iota}(\tau)P_{\iota}(\sigma)+\sum_{\iota=1}^{\zeta-3}\xi_{\iota}\xi_{\iota+1}
P_{\iota+1}(\tau)P_{\iota-1}(\sigma)\\
&+\sum_{\iota\geq\zeta-1}\psi_\iota(\si)\,P_\iota(\tau),
\end{split}
\end{equation*}
where
\begin{equation*}
\begin{split}
&\phi_\iota(\tau)=\alpha_\iota+\int_0^{\tau}\gamma_\iota(x)\,\dif
x,\quad\iota\geq\eta-1,\\
&\psi_\iota(\si)=\beta_\iota+\int_0^{\si}\lambda_\iota(x)\,\dif
x\quad\iota\geq\zeta-1.
\end{split}
\end{equation*}
We summarize the results above in the following lemma.
\begin{lem}\label{lemma:csRKN}
For the csRKN method \eqref{eq:csrkn1}-\eqref{eq:csrkn3} denoted by
$(\bar{A}_{\tau,\si},\bar{B}_\tau,B_\tau,C_\tau)$ with the
assumption $B_\tau=1, C_\tau=\tau$, we have the following
statements:
\begin{itemize}
\item[(I)] \begin{equation}\label{expan4}
\begin{split}
\mathcal{CN}(\eta)\;\Longleftrightarrow\;\bar{A}_{\tau,\,\sigma}&=\frac{1}{6}-\frac{1}{2}\xi_1P_1(\si)+\frac{1}{2}\xi_1P_1(\tau)
+\sum_{\iota=1}^{\eta-3}\xi_{\iota}\xi_{\iota+1}
P_{\iota-1}(\tau)P_{\iota+1}(\si)\\
&-\sum_{\iota=1}^{\eta-2}\big(\xi_{\iota}^2+\xi_{\iota+1}^2\big)
P_{\iota}(\tau)P_{\iota}(\sigma)+\sum_{\iota=1}^{\eta-1}\xi_{\iota}\xi_{\iota+1}
P_{\iota+1}(\tau)P_{\iota-1}(\sigma)\\
&+\sum_{\iota\geq\eta-1}\phi_\iota(\tau)\,P_\iota(\sigma),
\end{split}
\end{equation} where $\xi_\iota=\frac{1}{2\sqrt{4\iota^2-1}}\,(\iota\geq1)$ and
$\phi_\iota(\tau)\,(\iota\geq\eta-1)$ are arbitrary $L^2$-integrable
functions;
\item[(II)] \begin{equation}\label{expan5}
\begin{split}
\mathcal{DN}(\zeta)\;\Longleftrightarrow\;\bar{A}_{\tau,\,\si}&=\frac{1}{6}-\frac{1}{2}\xi_1P_1(\si)+\frac{1}{2}\xi_1P_1(\tau)
+\sum_{\iota=1}^{\zeta-1}\xi_{\iota}\xi_{\iota+1}
P_{\iota-1}(\tau)P_{\iota+1}(\si)\\
&-\sum_{\iota=1}^{\zeta-2}\big(\xi_{\iota}^2+\xi_{\iota+1}^2\big)
P_{\iota}(\tau)P_{\iota}(\sigma)+\sum_{\iota=1}^{\zeta-3}\xi_{\iota}\xi_{\iota+1}
P_{\iota+1}(\tau)P_{\iota-1}(\sigma)\\
&+\sum_{\iota\geq\zeta-1}\psi_\iota(\sigma)\,P_\iota(\tau),
\end{split}
\end{equation} where $\xi_\iota$ is defined as above, and
$\psi_\iota(\sigma)\,(\iota\geq\zeta-1)$ are arbitrary
$L^2$-integrable functions.
\end{itemize}
\end{lem}

\begin{thm}\label{thm:csRKN}
For the csRKN method \eqref{eq:csrkn1}-\eqref{eq:csrkn3} denoted by
$(\bar{A}_{\tau,\si},\bar{B}_\tau,B_\tau,C_\tau)$ with the
assumption $B_\tau=1, C_\tau=\tau$,  the following two statements
are equivalent to each other:
\begin{itemize}
\item[(I)] Both $\mathcal{CN}(\eta)$ and $\mathcal{DN}(\zeta)$ hold;
\item[(II)] The coefficient $\bar{A}_{\tau,\,\sigma}$
admits the following expansion:
\begin{equation}\label{expan6}
\begin{split}
\bar{A}_{\tau,\,\si}&=\frac{1}{6}-\frac{1}{2}\xi_1P_1(\si)+\frac{1}{2}\xi_1P_1(\tau)
+\sum_{\iota=1}^{N_1}\xi_{\iota}\xi_{\iota+1}
P_{\iota-1}(\tau)P_{\iota+1}(\si)\\
&-\sum_{\iota=1}^{N_2}\big(\xi_{\iota}^2+\xi_{\iota+1}^2\big)
P_{\iota}(\tau)P_{\iota}(\sigma)+\sum_{\iota=1}^{N_3}\xi_{\iota}\xi_{\iota+1}
P_{\iota+1}(\tau)P_{\iota-1}(\sigma)\\
&+\sum_{i\geq\zeta-1 \atop
j\geq\eta-1}\omega_{(i,\,j)}P_i(\tau)P_j(\si),
\end{split}
\end{equation}
where $N_1=\max\{\eta-3,\,\zeta-1\},\,N_2=\max\{\eta-2,\,\zeta-2\},
\,N_3=\max\{\eta-1,\,\zeta-3\},\,\xi_\iota=\frac{1}{2\sqrt{4\iota^2-1}}$
and $\omega_{(i,\,j)}$ are arbitrary real numbers.
\end{itemize}
\end{thm}
\begin{proof}
This theorem can be proved by using Lemma \ref{lemma:csRKN}. Let us
consider the expansions of $\phi_\iota(\tau)$ and
$\psi_\iota(\sigma)$
\begin{equation*}
\begin{split}
&\phi_\iota(\tau)=\sum_{i\geq0} \mu^\iota_i
P_i(\tau),\;\iota\geq\eta-1,\\
&\psi_\iota(\sigma)=\sum_{j\geq0} \nu^\iota_j
P_j(\sigma),\;\iota\geq\zeta-1,
\end{split}
\end{equation*}
where $\mu^\iota_i,\,\nu^\iota_j$ are real numbers. Considering that
$$P_i(\tau)P_j(\si),\;\;i,j=0,1,2,\cdots$$
form a complete orthogonal set in $L^2([0,1]\times[0,1])$, and
substituting the two expansions above into \eqref{expan4} and
\eqref{expan5} respectively, we then get the final result by
comparing similar terms.
\end{proof}

Since $\mathcal{B}(\infty)$ holds true for a csRKN method, Theorem
\ref{ord_csRKN} and \ref{thm:csRKN} allow to readily derive its
order as
$\min\{\infty,\,2\eta+2,\eta+\zeta\}=\min\{2\eta+2,\eta+\zeta\}$.

\begin{rem}\label{rem:csRKN_trunc}
For the sake of obtaining a practical csRKN method, we have to
define a finite form for $\bar{A}_{\tau,\,\si}$. A natural and
simple way is to truncate the series \eqref{expan6}. As a
consequence, the Butcher coefficient $\bar{A}_{\tau,\,\si}$ becomes
a bivariate polynomial in terms of $\tau$ and $\sigma$.
\end{rem}

\subsection{RKN methods by using quadrature formulas}

As for the practical implementation of the csRKN method
\eqref{eq:csrkn1}-\eqref{eq:csrkn3}, one has to approximate the
integrals by quadratures. Using a quadrature formula denoted by
$(b_i, c_i)_{i=1}^s$ yields an $s$-stage RKN method
\begin{subequations}
    \begin{alignat}{2}
    \label{eq:rkn1}
&Q_i=q_0 +hC_i q'_0+h^2\sum\limits_{j=1}^{s}b_j\bar{A}_{ij}f(t_0+C_jh,\,Q_j), \quad i=1,\cdots, s, \\
    \label{eq:rkn2}
&q_{1}=q_0+ h q'_0+h^2\sum\limits_{i=1}^{s}b_i\bar{B}_if(t_0+C_ih,\,Q_i), \\
    \label{eq:rkn3}
&q'_{1}= q'_0+h\sum\limits_{i=1}^{s}b_iB_if(t_0+C_ih,\,Q_i),
    \end{alignat}
\end{subequations}
where $\bar{A}_{ij}=\bar{A}_{c_i, c_j}, \bar{B}_i=\bar{B}_{c_i},
B_i=B_{c_i}, C_i=C_{c_i}$, which can be characterized by the
following Butcher tableau
\begin{equation}\label{RKN:qua_orig}
\ba{c|ccc} C_1 & b_1\bar{A}_{11}
& \cdots & b_s\bar{A}_{1s}\\[2pt]
\vdots &\vdots &\vdots\\[2pt]
C_s & b_1\bar{A}_{s1} &
\cdots & b_s\bar{A}_{ss}\\[2pt]
\hline & b_1\bar{B}_{1}  & \cdots & b_s\bar{B}_{s}\\[2pt]
\hline & b_1B_{1}  & \cdots & b_sB_{s}\ea
\end{equation}

Particularly, if we assume
$\bar{B}_\tau=B_\tau(1-C_\tau),\,B_\tau=1,\, C_\tau=\tau$ for
$\tau\in[0,1]$, then it gives an $s$-stage classical RKN method with
tableau
\begin{equation}\label{RKN:qua}
\ba{c|ccc} c_1 & b_1\bar{A}_{11}
& \cdots & b_s\bar{A}_{1s}\\[2pt]
\vdots &\vdots &\vdots\\[2pt]
c_s & b_1\bar{A}_{s1} &
\cdots & b_s\bar{A}_{ss}\\[2pt]
\hline & \bar{b}_1  & \cdots & \bar{b}_s\\[2pt]
\hline & b_1  & \cdots & b_s\ea
\end{equation}
where $\bar{b}_i=b_i(1-c_i),\; i=1,\cdots,s$. For the sake of
analyzing the order of the RKN method \eqref{RKN:qua}, linked with
Remark \ref{rem:csRKN_trunc}, we have the following result.

\begin{thm}\label{qua:csRKN}
Assume $\bar{A}_{\tau,\,\sigma}$ is a bivariate polynomial of degree
$\pi^{\tau}$ in $\tau$ and degree $\pi^{\sigma}$ in $\sigma$, and
the quadrature formula $(b_i,c_i)_{i=1}^s$ is of order $p$. If a
csRKN method \eqref{eq:csrkn1}-\eqref{eq:csrkn3} denoted by
$(\bar{A}_{\tau,\si},\bar{B}_\tau,B_\tau,C_\tau)$ with
$\bar{B}_\tau=B_\tau(1-C_\tau),\,B_\tau=1,\,
C_\tau=\tau,\,\tau\in[0,1]$ and both $\mathcal{CN}(\eta)$,
$\mathcal{DN}(\zeta)$ hold, then the RKN method \eqref{RKN:qua} is
at least of order
$$\min\{p, \,2\alpha+2, \,\alpha+\beta\},$$
where $\alpha=\min\{\eta,\,p-\pi^{\sigma}+1\}$ and
$\beta=\min\{\zeta,\, p-\pi^{\tau}+1\}$.
\end{thm}
\begin{proof}
Since $\int_0^1 g(x)\, \dif x=\sum_{i=1}^s b_i g(c_i)$ holds for any
polynomial $g(x)$ of degree up to $p-1$, by using the quadrature
formula $(b_i,c_i)_{i=1}^s$ to compute the integrals of
$\mathcal{B}(\xi)$, $\mathcal{CN}(\eta)$, $\mathcal{DN}(\zeta)$ one
obtains:
\begin{equation*}
\begin{split}
&\sum_{i=1}^sb_ic_i^{\kappa-1}=\frac{1}{\kappa},\;\kappa=1,\cdots,p,\\
&\sum_{j=1}^s(b_{j}\bar{A}_{ij})c_j^{\kappa-1}=\frac{c_i^{\kappa+1}}{\kappa(\kappa+1)},
\;i=1,\cdots,s,\;\kappa=1,\cdots,\alpha-1,\\
&\sum_{i=1}^sb_ic_i^{\kappa-1}(b_{j}\bar{A}_{ij})=
\frac{b_jc_j^{\kappa+1}}{\kappa(\kappa+1)}-\frac{b_jc_j}{\kappa}+\frac{b_j}{\kappa+1},
\;j=1,\cdots,s,\;\kappa=1,\cdots,\beta-1.
\end{split}
\end{equation*}
where $\alpha=\min\{\eta,\,p-\pi^{\sigma}+1\}$ and
$\beta=\min\{\zeta,\, p-\pi^{\tau}+1\}$.  These formulas imply that
the RKN method \eqref{RKN:qua} satisfies $B(p)$, ${CN}(\alpha)$ and
${DN}(\beta)$, and it is observed that $\bar{b}_i=b_i(1-c_i)$ is
naturally satisfied for each $i=1,\ldots, s$. Consequently, the
statement follows from Theorem \ref{ord_RKN}.
\end{proof}

\section{Construction of high-order symplectic integrators}

Hamiltonian systems constitute a very important subclass of
dynamical systems in the field of classical and non-classical
mechanics \cite{Arnold89mmo,Channels90sio,Feng84ods,hairerlw06gni,
sanzc94nhp,Feng95kfc,Fengqq10sga}. Such type of systems can be
written as
\begin{equation}\label{Hs}
z'=J^{-1}\nabla_{z}H(z),\;\;z(t_0)=z_0\in\mathbb{R}^{2d},\;\;
z=\begin{pmatrix}
  p \\
  q \\
\end{pmatrix},\;\;
J=\begin{pmatrix}
0 & I \\
-I & 0 \\
\end{pmatrix}.
\end{equation}
Symplectic integrators are of great interest for solving such
systems
\cite{Benetting94oth,Feng84ods,Feng95kfc,Fengqq10sga,Leimkuhlerr04shd,
sanzc94nhp,hairerlw06gni}, as they usually reproduce excellent
qualitative behaviors of the exact flow
\cite{Channels90sio,Shang99kam} and exhibit bounded energy errors
for exponentially-long time \cite{hairerlw06gni}.

In what follows, we restrict our attention to a special type of
Hamiltonian systems with the Hamiltonian function
$$H(z)=\frac{1}{2}p^TMp+V(q),$$ where $M$ is a constant symmetric
matrix, and $V(q)$ is a scalar function. Such systems are usually
called separable Hamiltonian systems, which reads
\begin{equation}\label{eq:first}
\begin{cases}
p'=-\nabla_q V(q),\\[2pt]
q'=Mp.
\end{cases}
\end{equation}
Substituting the second equation into the first equation gives
\begin{equation}\label{eq:Hs}
q''=-M\nabla_q V(q).
\end{equation}

Denote $f(q)=-M\nabla_q V(q)$ and $g(q)=-\nabla_q V(q)$, for solving
the equation \eqref{eq:Hs}, we propose the following csRKN method
\begin{subequations}\label{csRKN:Hs}
\begin{alignat}{2}
    \label{Heq:csrkn1}
&Q_\tau=q_0 +hC_\tau Mp_0 +h^2\int_{0}^{1} \bar{A}_{\tau, \si} f(Q_\si) \dif \si, \;\;\tau \in[0, 1], \\
    \label{Heq:csrkn2}
&q_{1}=q_0+ h Mp_0+h^2 \int_{0}^{1} \bar{B}_\tau  f(Q_\tau) \dif\tau, \\
    \label{Heq:csrkn3}
&p_1 = p_0 +h\int_{0}^{1} B_\tau g(Q_\tau) \dif \tau,
\end{alignat}
\end{subequations}
which is derived by replacing the variable $q'$ with $Mp$ in
Definition \ref{csRKN:def} but with $M$ eliminated in the last
formula. In \cite{Tangz18spc}, the authors have proved the following
results.

\begin{thm}\cite{Tangz18spc}\label{constr_symcsRKN}
The csRKN method denoted by
$(\bar{A}_{\tau,\si},\bar{B}_\tau,B_\tau,C_\tau)$ with $B_\tau=1,
C_\tau=\tau$ is symplectic, if $\bar{A}_{\tau,\si}$ and
$\bar{B}_\tau$ possess the following forms in terms of Legendre
polynomials
\begin{equation}\label{sym_cond}
\begin{split}
\bar{B}_\tau&=1-\tau=\frac{1}{2}P_0(\tau)-\xi_1P_1(\tau),\quad\tau\in[0,1],\\
\bar{A}_{\tau,\si}&=\alpha_{(0,0)}+\alpha_{(0,1)}P_1(\si)+\alpha_{(1,0)}P_1(\tau)+\sum\limits_{i+j>1}\alpha_{(i,j)}
P_i(\tau)P_j(\sigma),\quad\tau,\si\in[0,1],
\end{split}
\end{equation}
where $\alpha_{(0,0)}$ is an arbitrary real number,
$\alpha_{(0,1)}-\alpha_{(1,0)}=-\xi_1=-\frac{\sqrt{3}}{6}$, and
$\alpha_{(i,j)}=\alpha_{(j,i)},\;\forall\,i+j>1$.
\end{thm}

\begin{thm}\cite{Tangz18spc}\label{sym_quad}
If the csRKN method denoted by
$(\bar{A}_{\tau,\si},\bar{B}_\tau,B_\tau,C_\tau)$ satisfies the
symplecticity conditions
\begin{subequations}
\begin{alignat}{2}
\label{sym_cond_orig01}
\bar{B}_\tau&=B_\tau(1-C_\tau),\quad\tau\in[0,1],\\
\label{sym_cond_orig02}
B_\tau(\bar{B}_\si-\bar{A}_{\tau,\si})&=B_\si(\bar{B}_\tau
-\bar{A}_{\si,\tau}),\quad\tau,\si\in[0,1],
\end{alignat}
\end{subequations}
then the associated RKN method \eqref{RKN:qua_orig} derived by using
a quadrature formula $(b_i, c_i)_{i=1}^s$ is symplectic.
\end{thm}

Instead of constructing symplectic integrators step by step with
each step getting one higher order as in \cite{Tangz18spc}, now we
can directly use Theorem \ref{thm:csRKN}. One just needs to compare
the two series in terms of Legendre polynomials as shown in
\eqref{expan6} and \eqref{sym_cond}, and then devise symplectic
csRKN integrators by a suitable truncation of the series. Theorem
\ref{sym_quad} implies that one can use quadrature formulas to get
standard symplectic RKN methods from symplectic csRKN methods, and
the order of the resulting methods can be analyzed by Theorem
\ref{qua:csRKN}. It turns out that all the symplectic integrators
presented in \cite{Tangz18spc} can be recovered by using this new
technique. Moreover, new symplectic integrators can be established
in a more convenient way, for instance, if we take
$\bar{B}_\tau=1-\tau,\;B_\tau=1,\;C_\tau=\tau$ and
\begin{equation}\label{eq:6coeff}
\begin{split}
\bar{A}_{\tau,\,\si}&=\frac{1}{6}-\frac{1}{2}\xi_1P_1(\si)+\frac{1}{2}\xi_1P_1(\tau)
+\sum_{\iota=1}^{2}\xi_{\iota}\xi_{\iota+1}
P_{\iota-1}(\tau)P_{\iota+1}(\si)\\
&-\big(\xi_{1}^2+\xi_{2}^2\big)
P_{1}(\tau)P_{1}(\sigma)+\sum_{\iota=1}^{2}\xi_{\iota}\xi_{\iota+1}
P_{\iota+1}(\tau)P_{\iota-1}(\sigma)+\theta P_2(\tau)P_2(\si),
\end{split}
\end{equation}
where $\xi_\iota=\frac{1}{2\sqrt{4\iota^2-1}}$ and $\theta$ is an
arbitrary real parameter, then, the use of the $3$-point Gaussian
quadrature, produces a family of $3$-stage $6$-order symplectic RKN
methods, having Butcher tableau
\[\ba{c|ccc} \frac{5-\sqrt{15}}{10} & \frac{2+30\theta}{135} &
\frac{19-6\sqrt{15}-120\theta}{270}&
\frac{62-15\sqrt{15}+120\theta}{540}\\[2pt]
\frac{1}{2} & \frac{19+6\sqrt{15}-120\theta}{432}&
\frac{1+15\theta}{27}&
\frac{19-6\sqrt{15}-120\theta}{432}\\[2pt]
\frac{5+\sqrt{15}}{10}& \frac{62+15\sqrt{15}+120\theta}{540}
&\frac{19+6\sqrt{15}-120\theta}{270}&\frac{2+30\theta}{135}\\[2pt]
\hline & \frac{5+\sqrt{15}}{36} & \frac{2}{9} & \frac{5-\sqrt{15}}{36}\\[2pt]
\hline & \frac{5}{18} & \frac{4}{9} & \frac{5}{18} \ea\]

\section{Concluding remarks}

In this paper, we present a new technique to construct high-order
symplectic integrators via analyzing the simplifying assumptions for
order conditions by means of orthogonal expansions. The new
technique shows that high-order integrators can be devised by
truncating an orthogonal series. Besides, one could introduce some
free parameters in the formulation of Butcher coefficients --- an
underlying application for this is that one may get explicit or
semi-explicit integrators with suitable choices of these parameters
(see \cite{Tangz18spc}).

\section*{Acknowledgements}

The first author was supported by the National Natural Science
Foundation of China (11401055), China Scholarship Council
(No.201708430066) and Scientific Research Fund of Hunan Provincial
Education Department (15C0028). The second author was supported by
the Foundation of the NNSFC (No.11271357), the Foundation for
Innovative Research Groups of the NNSFC (No.11321061) and ITER-China
Program (No.2014GB124005). The third author was supported by the
foundation of NNSFC (No. 11201125, 11761033) and PhD scientific
research foundation of East China Jiaotong University.

\end{document}